\documentclass[10pt]{amsart}
\usepackage{amsmath,amssymb}
\usepackage[sc]{mathpazo}
\usepackage{amsthm,stackrel,bm}

\usepackage{setspace}
\usepackage[pdfborderstyle={/S/U/W 1}]{hyperref}

\usepackage{amsrefs}

\newtheorem{theorem}{Theorem}
\newtheorem{lemma}{Lemma}

\theoremstyle{definition}

\newtheorem{remark}{Remark}

\numberwithin{lemma}{section}

\DeclareMathOperator{\sgn}{sgn}

\renewcommand{\mod}[1]{\ensuremath{\left( \mathrm{mod} \ #1 \right)}}

\newcommand{\h}{\ensuremath{\mathrm{h}}}
\newcommand{\upd}{\ensuremath{\mathrm{d}}}
\newcommand{\bbR}{\ensuremath{\mathbb{R}}} 
\newcommand{\RR}{\ensuremath{\mathbb{R}}} 
\newcommand{\CC}{\ensuremath{\mathbb{C}}}
\newcommand{\veps}{\ensuremath{\varepsilon}}

\newcommand{\defeq}{\ensuremath{\stackrel{\mathrm{def}}{=}}} 
\newcommand{\To}{\ensuremath{\longrightarrow}} 
\newcommand{\g}{\ensuremath{\mathrm{g}}}
\newcommand{\prtl}{\ensuremath{\partial}}
\newcommand{\cone}{\ensuremath{C(\mathbb{S}^1_\rho)}}
\DeclareMathOperator{\supp}{supp}

\let\Re\relax
\DeclareMathOperator{\Re}{Re}

\begin{document}

\title[Spectral clusters on polygonal domains]{$L^p$-bounds on spectral clusters
  associated to polygonal domains}

\author[M.D.~Blair]{Matthew D.\ Blair} \address{Department of Mathematics and Statistics, University of New Mexico, Albuquerque, NM 87131, USA}
\email{blair@math.unm.edu}

\author[G.A.~Ford]{G.~Austin Ford} \address{Department of Mathematics, Stanford University, Palo Alto, CA  94305, USA}
\email{austin.ford@math.stanford.edu}

\author[J.L.~Marzuola]{Jeremy L.\ Marzuola} \address{Department of Mathematics, University of North
  Carolina, Chapel Hill, NC 27599, USA} \email{marzuola@math.unc.edu}

\begin{abstract}
  We look at the $L^p$ bounds on eigenfunctions for polygonal domains (or more
  generally Euclidean surfaces with conic singularities) by analysis of the
  wave operator on the flat Euclidean cone $C(\mathbb{S}^1_\rho) \defeq
  \mathbb{R}_+ \times \left(\mathbb{R} \big/ 2\pi\rho \mathbb{Z}\right)$ of
  radius $\rho > 0$ equipped with the metric $\h(r,\theta) = \upd r^2 + r^2 \,
  \upd\theta^2$.  Using explicit oscillatory integrals and relying on the fundamental solution to
  the wave equation in geometric regions related to flat wave propagation
  and diffraction by the cone point, we can prove spectral cluster estimates equivalent to those in works on smooth Riemannian manifolds.  
\end{abstract}

\maketitle

\onehalfspace

\section{Introduction}
\label{sec:introduction}

Let $\Omega \subset \RR^2$ be a compact polygonal domain in the
plane, that is, a compact, connected region in $\RR^2$ whose
boundary, $\partial \Omega$, is piecewise linear.  Note, we place no restrictions here on the polygon in terms of convexity or rationality.  Suppose $\{ \phi_j \}$, $\phi_j : \Omega \To
\CC$ is an orthonormal $L^2 (\Omega)$ eigenbasis for the (positive) Laplacian operator on $\Omega$ with either Dirichlet or Neumann boundary conditions on $\partial \Omega$, 
\begin{equation}
  \label{eqn:laplace}
    \begin{aligned}
       \Delta \phi_j &=  \lambda_j^2 \phi_j \ \ 0 \leq \lambda_0 < \lambda_1 \leq \cdots \leq \lambda_j \leq \lambda_{j+1} \leq \cdots, \ \ \| \phi_j \|_{L^2 (\Omega)} = 1.
    \end{aligned}
\end{equation}

We study $L^p$ boundedness properties of the $\phi_j$'s depending upon their frequency, which can be achieved by proving estimates on clusters of eigenfunctions.  There is a rich history of spectral cluster estimates on smooth, closed Riemannian manifolds, classically going back to the work of Avakumovi\v{c}, Levitan, and H\"ormander and more recently in the work of Sogge \cite{sogge}, with many further extensions to manifolds with boundary such as \cite{grieser1992lp,smith1994lp,smith2007p,blair2014strichartz}.  Other extensions to metrics of less regularity can be found in for instance \cite{smith,KST,blair-spclus}.
However, the estimates in the present work appear to be the first on domains with corners or conic singularities except for rectangles.  See for instance the recent work of Bourgain-Demeter \cite{bourgain2014proof}, where restriction estimates on general tori are studied.  Indeed, $L^p$ bounds on the eigenfunctions can be viewed via the Stein-Tomas restriction theorem as a version of the adjoint restriction estimate on the sphere.  The authors have previously treated the analogs of adjoint restriction estimates for polygonal domains in cases of the parabola in \cite{For,BFHM} and the cone in \cite{BFM} by proving Strichartz estimates for the Schr\"odinger equation and wave equation respectively in the setting polygonal domains.  Arguably, the sphere presents unique challenges since Strichartz bounds for the Schr\"odinger and wave equations rely only on fixed time bounds for the corresponding kernel, whereas the spectral cluster bounds typically require integrating/averaging the wave kernel and estimating the contributions of the jumps in the transition from geometric to diffracted wave fronts.

\begin{remark}
The Neumann Laplacian on $\Omega$ is taken to be the Friedrichs extension of the Laplace operator acting on smooth functions which vanish in a neighborhood of the vertices and whose normal derivative is zero on the rest of the boundary. 
The Dirichlet Laplacian is taken to be the typical Friedrichs extension of the Laplace operator acting on smooth functions which are compactly supported in the interior of $\Omega$.
\end{remark}


The spectral projection operator $\Pi_\lambda$ is defined for any $\lambda \geq 0$ such that
\begin{equation}
\label{eqn:specproj}
\Pi_\lambda f = \sum_{\lambda_j \in [\lambda, \lambda +1 ]}  \langle f , \phi_j \rangle \phi_j.
\end{equation}
We refer to functions in the range of $\Pi_\lambda$ as "spectral clusters."  Then, the desired spectral cluster estimates are stated as the following theorem.

\begin{theorem}
\label{conj:spclus}
For $\Omega$ any polygonal domain in $\RR^2$, $f \in L^2 (\Omega)$, we have
\begin{equation}
\label{e:spclus}
\| \Pi_\lambda f \|_{L^q (\Omega)} \leq C \lambda^{\delta(q)} \| f \|_{L^2 (\Omega)}, \ \ \delta(q) = \left\{
\begin{array}{l}
\frac{1}{2} \left( \frac12 - \frac{1}{q} \right) \ \ \text{for} \ \ 2 \leq q \leq 6, \\
2 \left( \frac12 - \frac1{q} \right) - \frac12 \ \ \text{for} \ \ 6 \leq q \leq \infty
\end{array} \right.
\end{equation}
for $C$ independent of $\lambda \geq 1$.  Consequently, given any $L^2$ normalized eigenfunction $\Delta\phi_\lambda=\lambda^2 \phi_\lambda$ we have
\[
\|\phi_\lambda\|_{L^q(\Omega)} \leq C \lambda^{\delta(q)}.
\]
\end{theorem}

As in \cite{BFHM,BFM}, we will in reality establish Theorem \ref{conj:spclus} for a Euclidean surface with conical singularities (ESCS).  When the cone angle is a rational multiple of $\pi$, this has a special type of orbifold structure.  An ESCS is a Riemannian surface $(X,\g)$ that can be covered by a finite number of coordinate charts, each of which is isometric to a subset of $\bbR^2$ or $\cone$.   Let $\cone$ denote the Euclidean cone of radius
$\rho>0$, defined as the product manifold
$\cone=\RR_+ \times\left(\mathbb{R} \big/ 2\pi\rho \mathbb{Z}\right)$, equipped
with the metric $\g(r,\theta) = \upd r^2 + r^2 \, \upd\theta^2$.  This is an
incomplete manifold which is locally isometric to $\RR^2$ away from the cone points and hence flat.  For a more precise definition, see \cite{BFHM}.  Even though the manifolds we consider have conic singularities, the power $\delta (q)$ that appears in Theorem \ref{conj:spclus} is the same as that in Sogge's original estimates for spectral clusters  on $C^\infty$ manifolds \cite{sogge}.  The same work shows that this is the sharp exponent for spectral clusters on any Riemannian manifold, though this exponent may not be optimal for individual eigenfunctions.

Any compact planar polygonal domain $\Omega$ can be doubled across its boundary to produce a compact ESCS.  In this procedure, a vertex of $\Omega$ of angle $\alpha$ gives rise to a conic point of $X$ with cone angle $2\alpha$.  We then take the Laplace-Beltrami operator, $\Delta_\g$ on $X$ to be the Friedrichs extension of the Laplacian on $\mathcal{C}^\infty_c(X_0)$, where $X_0$ is $X$ less the singular points.    To see this clearly, let us recall the procedure outlined in  \cite[Section 2]{BFHM}.  Begin with two copies $\Omega$ and $\sigma \Omega$ of the polygonal domain, where $\sigma$ is a reflection of the plane.  An ESCS $X$ is then obtained by taking the formal union $\Omega \cup \sigma \Omega$, where two corresponding sides are identified pointwise.  Taking polar coordinates near each vertex of the polygon, it can be seen that the flat metric $\g$ extends smoothly across the sides.  In particular, a vertex in $\Omega$ of angle $\alpha$ gives rise to a conic point of $X$ locally isometric to $C(\mathbb{S}^1_\rho)$ with $\rho = \frac{\alpha}{\pi}$.  Such a doubling procedure produces a conic point of angle $2 \alpha$.

The reflection $\sigma$ of $\Omega$ gives rise to an involution of $X$ commuting with the Laplace-Beltrami operator.  We thus have a decomposition into two operators acting on functions which are either odd (even) with respect to $\sigma$, which are equivalent to the Laplace operator on $\Omega$ with Dirichlet (Neumann) boundary conditions respectively.  For us, the key observation is that for any eigenfunction $\varphi_j$ of the Dirichlet, resp.\ Neumann, Laplace operator on $\Omega$, we can construct an eigenfunction of the Laplace operator on $X$ by taking $\varphi_j$ in $\Omega$ and $-\varphi_j \circ \sigma$, resp. $\varphi_j \circ \sigma$, in $\sigma \Omega$.  As a consequence, the spectrum over $X$ can be seen to extend that for $\Omega$.     See the previous works of the authors \cite[Section 2]{BFHM} for a thorough description of ESCSs and $\Delta_\g$, as well as \cite[Section 2]{BFM} for a general treatment of Cheeger's functional calculus on cones.  

Theorem \ref{conj:spclus} then follows from the equivalent statement for ESCSs.
\begin{theorem}
  \label{conj:spclusX}
  For $X$ any compact ESCS, $f \in L^2 (X)$, we have
  \begin{equation}
    \label{e:spclusX}
    \| \Pi_\lambda f \|_{L^q (X)} \leq C \lambda^{\delta(q)} \| f \|_{L^2 (X)}, \ \ \delta(q) = \left\{
      \begin{array}{l}
        \frac{1}{2} \left( \frac12 - \frac{1}{q} \right) \ \ \text{for} \ \ 2 \leq q \leq 6, \\
        2 \left( \frac12 - \frac1{q} \right) - \frac12 \ \ \text{for} \ \ 6 \leq q \leq \infty
      \end{array} \right.
  \end{equation}
  for  $C$ independent of $\lambda$.
\end{theorem}

\subsection{Obtaining Spectral Cluster Estimates}
\label{sec:spectral-cluster}

As is well understood and explored below (see also the result from \cite{Sog}),
Theorem \ref{conj:spclus} can be related to forming an oscillatory integral which integrates the wave kernel in time on the Euclidean cone.    In order to pursue such estimates, we will consider the fundamental solution of the wave equation on the Euclidean cone,
\begin{equation}
  \label{eqn:wave}
  \left\{
    \begin{aligned}
      \left(D_t^2 - \Delta_\g \right) u(t,r,\theta) &= F \\
      u(0,r,\theta) &= f(r,\theta) \\
      \prtl_t u (0,r,\theta) &= g(r,\theta) \\
    \end{aligned}
  \right.
\end{equation}
for $u:\RR \times \cone \to \RR$.
A pioneering work regarding the fundamental solution to the wave equation on manifolds with conic singularities is
that of Cheeger and Taylor~\cite{CheTay1,CheTay2} who studied the propagation of
singularities for solutions amongst other properties.  Further progress on the
regularity of the fundamental solution was made by Melrose and Wunsch
in~\cite{melrosewunsch}.  Let us recall from \cite{CheTay2}, Section $4$ and \cite{BFM}, Equations $(3.14)-(3.16)$ that the wave fundamental solution kernel for $\sin(t\sqrt{ \smash[b]{ \Delta_g } })/\sqrt{ \smash[b]{ \Delta_g } }$ on the cone can be written as a decomposition of a geometric component,
\begin{multline}
\label{geowk}
K^{\text{geom}} (t,r_1,\theta_1;r_2,\theta_2)  = \\
 \sum_{-\pi \leq (\theta_1 -\theta_2) +j \cdot 2 \pi \rho \leq
    \pi}\frac{1}{(t^2 -r_1^2-r_2^2+2r_1r_2 \cos ((\theta_1-\theta_2) +j \cdot 2 \pi \rho)
    )^{\frac 12}_+} ,
\end{multline}
and a diffracted component,
\begin{multline}\label{diffwk}
K^{\text{diff}}(t, r_1, \theta_1;r_2, \theta_2) = 
  -\frac{\mathbf{1}_{(0,t)}(r_1+r_2)}{4\pi^2 \rho(2r_1r_2)^{\frac 12}} \\
  \mbox{} \times
  \int_0^\beta \left( \alpha -\cosh s \right)^{-\frac 12}\left[ \frac{\sin
      \varphi_1}{\cosh(s/\rho)-\cos \varphi_1} + \frac{\sin
      \varphi_2}{\cosh(s/\rho)-\cos \varphi_2} \right]\;ds,
\end{multline}
where we have used the abbreviations
\begin{equation*}
  \alpha = \frac{t^2-r_1^2-r_2^2}{2r_1 r_2} = \frac{t^2-(r_1+r_2)^2}{2r_1 r_2}
  +1, \qquad \qquad \beta = \cosh^{-1}(\alpha),
\end{equation*}
\begin{equation*}
  \varphi_1 = \frac{\pi + \theta_1 -\theta_2}{\rho}, \qquad \varphi_2 =
  \frac{\pi -(\theta_1 -\theta_2)}{\rho}.
\end{equation*}

\begin{remark}
As shown in \cite{burq2008global,smith2007p,smith}, spectral cluster estimates are equivalent to proving a dispersive estimate that holds
on the representative geometry of each coordinate patch of the domain
$\Omega$.  Namely, using Fourier analysis in the $t$-variable, spectral cluster estimates can be related to dispersive estimates for a solution to the wave equation on an ESCS, $X$,
\begin{equation}
  \label{eqn:waveX}
  \left\{
    \begin{aligned}
      \left(D_t^2 - \Delta_\g \right) u(t,x) &= F \\
      u(0,x) &= f(x) \\
      \prtl_t u (0,x) &= g(x). \\
    \end{aligned}
  \right.
\end{equation}
To be more precise, Theorem \ref{conj:spclusX} on a Riemannian
manifold, $M$, is equivalent to the dispersive-type estimate
\begin{equation}
  \label{e:sqfuncest1}
  \| u \|_{L^q_x ( M; L^2_t [-T,T])} \leq C \left( \| (f,g) \|_{H^{\delta(q)} \times H^{\delta(q)-1}} + \| F \|_{L^1_t ( [-T,T]; H^{\delta (q) -1) )}} \right)
\end{equation}
for $u$ a solution to \eqref{eqn:waveX}, see \cite{smith,burq2008global}.
Note that these estimates are typically associated with a measure of decay away from the light cone and hence
differ in form from the standard Strichartz estimates which capture dispersive
decay.  See \cite{BFM} for more on Strichartz estimates in this setting as well.  In addition, $L^p$ regularity for wave operators on product cones and their applications to $L^p$ bounds for spectral multipliers have been studied in \cite{mulseeg}.
 \end{remark}

The proof of Theorem \ref{conj:spclusX} will follow once we derive proper representations of the
spectral projection operators as oscillatory integrals.
One proof of \eqref{e:spclus} on $\RR^2 $ begins by first observing (cf. p.130, 137 in \cite{Sog}) that one may
replace $\Pi_\lambda$ by $\chi(\sqrt{ \smash[b]{ \Delta_g } }-\lambda)$ with
$\chi \in \mathcal{S}(\RR)$ even and real-valued, with $\chi >0$ in a neighborhood of 0, and
$\supp(\widehat{\chi}) \subset \{ |t| \in (\delta,2\delta) \}$ for some $\delta >0$.  Note, here we are considering the wave operator on $\RR^2$ but similar approaches work on more general manifolds.  It can
then be seen that the Schwartz kernel of $\chi(\sqrt{\Delta}-\lambda) $ is a
convolution kernel, which as a function of $z$ is of the form
\begin{equation}\label{convokernel1}
  \lambda^{\frac 12} \sum_\pm e^{\pm i\lambda|z|}a_{\lambda,\pm}(|z|) + R_\lambda(z)
\end{equation}
where $a_{\lambda,\pm}(\cdot)$ is compactly supported in $(\delta/2, 4\delta)$
and $R_\lambda(z)$ satisfies much better bounds than is needed:
$|\prtl^\alpha_z R_\lambda(z)| \lesssim_{N,\alpha} \lambda^{-N}$.  The phase
function $|x-y|$ is a Carleson-Sj\"olin phase, so the desired
$L^2(\RR^2) \to L^6(\RR^2)$ bounds then follow from oscillatory integral estimates in \cite{HorOsc}.   For a generalization of this result to higher dimensions, see for instance
Stein's variable coefficient
generalization of the Stein-Tomas restriction theorem (see e.g. Corollary 2.2.3
in \cite{Sog}).

The easiest way to see \eqref{convokernel1} is to write the Schwartz kernel as a
Fourier integral in polar coordinates
$$
\int_0^\infty \left(\int_0^{2\pi} e^{irz\cdot \theta} d\theta \right)
\chi(r-\lambda)r\,dr.
$$
Stationary phase shows that
$$
\int_0^{2\pi} e^{irz\cdot \theta} d\theta =|rz|^{-\frac 12} \sum_\pm e^{\pm
  ir|z|}a_\pm(r|z|)
$$
where $a_\pm$ are smooth and bounded.  When $|z| \in (\delta/2, 4\delta)$,
\eqref{convokernel} follows by using that the fact that $\chi$ is Schwartz
allows one to essentially replace $r$ by $\lambda$.  Seeing the rapid decay in
$\lambda$ when $|z| \notin (\delta/2, 4\delta)$ takes some extra work.  In short,
one has to replace $\chi$ by its Fourier transform, but we will see it by a
different method below in Section \ref{sec:clust-estim-polyg}.  


Such a representation of the fundamental solution generally allows one to
establish the $L^6$ bounds we desire.  In the case of the geometric
wave, we will observe that the leading order fundamental solution representation has the correct form of a Carleson-Sj\"olin phase, and
the result holds from standard arguments.  The diffracted component presents a different challenge in that the phase function is not of the desired form, thus we need a modified argument to get the correct decay.

\subsection*{Acknowledgement.} MDB is supported by NSF grant DMS-1301717.  GAF is supported by NSF Postdoctoral Fellowship grant DMS-1204304. JLM was supported by NSF Grant DMS--1312874 and NSF CAREER Grant DMS--1352353.  The authors are grateful to Andrew Hassell for helpful conversations and to Tadahiro Oh for pointing out the importance of $L^\infty$ eigenfunction estimates in establishing Gibbs measures, which led the authors down the path of beginning to prove spectral cluster estimates as a first step towards such a goal.



\section{Spectral cluster estimates on polygonal domains}
\label{sec:clust-estim-polyg}

\subsection{Treatment of the geometric term}
Let $X$ be an ESCS of dimension 2. We are interested in establishing the bound
\begin{equation}
\label{specpbds}
\|\Pi_\lambda\|_{L^2(X) \to L^p(X)} \lesssim \lambda^{\max(\frac 14-\frac 1{2p},\frac 12-\frac 2p)}
\end{equation}
where $\Pi_\lambda$ projects onto eigenspaces corresponding to frequencies $\lambda_j$ satisfying $\lambda_j \in [\lambda, \lambda+1]$.  Note that this is a discrete analog of the Fourier multiplier determined by the symbol $\mathbf{1}_{[\lambda, \lambda+1]}(\xi)$ on $\RR^2$.  As noted above, (cf. \cite[p.130, 137]{Sog}) that it suffices to prove this replacing $\Pi_\lambda$ by $\chi(\sqrt{ \smash[b]{ \Delta_g } }-\lambda)$, where $\chi \in \mathcal{S}(\RR)$, is even and real valued and  $\chi >0$ in a neighborhood of 0 with $\supp(\widehat{\chi}) \subset \{|t| \in(\delta, 2\delta)\}$ for some $\delta >0$. Hence
\begin{align*}
\chi(\sqrt{ \smash[b]{ \Delta_g } }-\lambda) &= \frac{1}{2\pi}\int e^{it(\sqrt{ \smash[b]{ \Delta_g } }-\lambda)}\widehat{\chi}(t)\,dt\\
& = \frac{1}{\pi}\int_{-\infty}^\infty e^{-it\lambda}\cos(t\sqrt{ \smash[b]{ \Delta_g } })\widehat{\chi}(t)\,dt +\tilde{\chi}(\sqrt{ \smash[b]{ \Delta_g } }+\lambda)
\end{align*}
where $\tilde{\chi}$ is some other Schwartz class function.  Since the spectrum of $\sqrt{ \smash[b]{ \Delta_g } }$ is positive, $\tilde{\chi}(\sqrt{ \smash[b]{ \Delta_g } }+\lambda)$ is a rapidly decaying function of an elliptic operator, and hence $\|\tilde{\chi}(\sqrt{ \smash[b]{ \Delta_g } }+\lambda)\|_{L^2(X)\to L^p(X)} =O(\lambda^{-N})$ for any $N>0$. Consequently it suffices to restrict attention to the operator valued integral here.

Integration by parts yields
\begin{multline}\label{sineformula}
\int_{-\infty}^\infty e^{-it\lambda}\cos(t\sqrt{ \smash[b]{ \Delta_g } })\widehat{\chi}(t)\,dt \\=
i\lambda\int_{-\infty}^\infty e^{-it\lambda}\frac{\sin(t\sqrt{ \smash[b]{ \Delta_g } })}{\sqrt{ \smash[b]{ \Delta_g } }}\widehat{\chi}(t)\,dt
-\int_{-\infty}^\infty e^{-it\lambda}\frac{\sin(t\sqrt{ \smash[b]{ \Delta_g } })}{\sqrt{ \smash[b]{ \Delta_g } }} \widehat{\chi}'(t)\,dt.
\end{multline}
By Sobolev embedding, the operator defined second term here satisfies stronger $L^2(X)\to L^p(X)$ bounds than needed, so we may also neglect its contribution. We further note that since $\widehat{\chi}$ is even, the first term on the right can be rewritten as
\begin{equation}\label{evenintegral}
\lambda\int_{-\infty}^\infty \sin(t\lambda)\frac{\sin(t\sqrt{ \smash[b]{ \Delta_g } })}{\sqrt{ \smash[b]{ \Delta_g } }}\widehat{\chi}(t)\,dt
=2\lambda\int_{0}^\infty \sin(t\lambda)\frac{\sin(t\sqrt{ \smash[b]{ \Delta_g } })}{\sqrt{ \smash[b]{ \Delta_g } }}\widehat{\chi}(t)\,dt.
\end{equation}
By finite speed of propagation, the Schwartz kernel of this operator thus vanishes when the distance between the two points on $X$ is larger than $2\delta$.  See \cite{BFM}, $(1.16)$ or \cite{CheTay1}, $(3.41)$ for a complete definition of this notion of distance on the cone.  Consequently, it suffices to prove $L^2 \to L^p$ bounds for data supported in a chart where $X$ can be identified with a flat cone, $\cone$.  Moreover, using the fact that the wave kernels respect periodicity, by a doubling argument, if the bounds hold when the radius is $\rho$, then they also hold when the radius is $\rho/2$.  We may thus assume that $\rho>1$ (recalling that the $\rho=1$ follows simply by identification with $\RR^2$, see the treatment below).

We finally remark that it suffices to establish $p=\infty$ and $p=6$ bounds on the operator in \eqref{evenintegral} as the remaining bounds will follow from interpolation.

\subsubsection{The Schwartz kernel of \eqref{evenintegral} on $\RR^2$} We begin by computing the Schwartz kernel of the operator in \eqref{evenintegral} when $X=\RR^2$ and $\Delta_g=\Delta$ is the standard Laplacian on $\RR^2$.  While this can be accomplished by employing the methods in \cite{Sog}, we include an alternative presentation as it can be used to help treat the ``geometric" contribution below.  In particular, we will only use that the fundamental solution of the wave equation is of the form $\frac{1}{2\pi}(t^2-|z|^2)_+^{-\frac 12}$.  It will be seen that the Schwartz kernel for the integral in \eqref{evenintegral}  is a convolution kernel, which as a function of $z$, is of the form
\begin{equation}\label{convokernel}
\Re \left(\lambda^{\frac 12} e^{i\lambda|z|}a_{\lambda}(|z|) \right)+ R_\lambda(z)
\end{equation}
where $a_{\lambda}(\cdot)$ is compactly supported and smooth in $(\delta, 2\delta)$ and $R_\lambda(z)$ satisfies stronger $L^2(X) \to L^p(X)$ bounds.   In particular, $R_\lambda(z)$ is $O(\lambda^{-N})$ for any $N$.   The bounds when $p=\infty$ are then immediate.  Moreover, the phase function $|x-y|$ is a Carleson-Sj\"olin phase, so the desired $L^2(X) \to L^6(X)$ bounds then follow from H\"ormander \cite{HorOsc} as stated in the Introduction.

The kernel of the operator in \eqref{evenintegral} is a convolution kernel, and neglecting harmless constants, this as a function of $z$ is given by:
\begin{align}\label{fourierfund}
\lambda\int_{|z|}^\infty \sin(t\lambda)&(t^2-|z|^2)^{-\frac 12}_+\widehat{\chi}(t)\,dt  \\
&=\lambda\int_{-\infty}^\infty \chi(\tau)\int_{|z|}^\infty \sin(t\lambda)\cos(t\tau)(t^2-|z|^2)^{-\frac 12}_+\,dt d\tau\notag\\
&=\frac{\lambda}{2}\int_{-\infty}^\infty \chi(\tau)\int_{|z|}^\infty \left(\sin(t(\lambda+\tau))+\sin(t(\lambda-\tau))\right)(t^2-|z|^2)^{-\frac 12}_+\,dt d\tau \notag\\
&=
\lambda\int_{-\infty}^\infty \chi(\tau)\int_{|z|}^\infty \sin(t(\lambda-\tau))(t^2-|z|^2)^{-\frac 12}_+\,dt d\tau. \notag
\end{align}
Here the second equality follows from trigonometric identities and the first and third equalities use that $\chi(\tau)$ is even.  Now observe that after a change of variables $t=s|z|$ and the identities \cite[p.170]{Wat}, we have
\begin{equation}\label{besselidentity}
\frac{2}{\pi}\int_{|z|}^\infty \sin(t\zeta)(t^2-|z|^2)^{-\frac 12}_+\,dt=\frac{2}{\pi}\int_{1}^\infty \sin(s|z|\zeta)(s^2-1)^{-\frac 12}_+\,ds=
\sgn(\zeta) J_0(|z\zeta|)
\end{equation}
where $J_0$ is the Bessel function of order 0.  Neglecting harmless constants once again, we are now led to consider
\begin{equation}\label{besselconvo}
\lambda\int_{-\infty}^\infty \chi(\tau)\sgn(\lambda-\tau)J_0(|z||\lambda-\tau|) \,d\tau .
\end{equation}

Let $\psi$ be a smooth, even bump function such that $\supp(\psi) \subset (-\frac 12,\frac 12)$ and $\supp(1-\psi) \subset (-\frac 14, \frac 14)^c$, and observe that
\[
\lambda\int_{-\infty}^\infty \chi(\tau)(1-\psi)(\lambda^{-1}\tau)\sgn(\lambda-\tau)J_0(|z||\lambda-\tau|) \,d\tau = O(\lambda^{-N})
\]
for any $N>0$ given that $\chi(\tau)$ is a Schwartz class function and $J_0$ is bounded.  Note that this relation is uniform in $z$.  Consequently, since $\lambda -\tau \geq \lambda/2$ when $\lambda^{-1} \tau \in \supp(\psi)$, we may restrict attention to the contribution of
\begin{equation}\label{besselconvomain}
\lambda\int_{-\infty}^\infty \chi(\tau)\psi(\lambda^{-1}\tau)J_0(|z|(\lambda-\tau)) \,d\tau .
\end{equation}

Typical stationary phase arguments imply that for $\zeta \in (0,\infty)$,
\[
J_0(\zeta) = \Re\left( e^{i\zeta}b(\zeta) \right),
\]
where the $k$-th derivative of $b$ satisfies
\[
|b^{(k)}(\zeta)| \lesssim_k (1+\zeta)^{-k-\frac 12}, \qquad k \geq 0.
\]
We thus rewrite \eqref{besselconvomain} as
\begin{equation}\label{besselintegration}
\Re\left(\lambda e^{i\lambda|z|}\int_{-\infty}^\infty e^{-i\tau|z|} \chi(\tau)\tilde{\psi}(\lambda^{-1}\tau,\lambda|z|)\,d\tau \right).
\end{equation}
where for $\tau \in \RR$ and $w\geq 0$,
\[
\tilde{\psi}(\tau,w) \defeq b(w(1-\tau))\psi(\tau).
\]
Recalling that $\supp(\psi) \subset (-\frac 12,\frac 12)$, we have
\[
| \prtl^{j}_\tau \prtl^{k}_w \tilde{\psi}(\tau,w) | \lesssim_{j,k} (1+w)^{-\frac 12 -k} \qquad j,k \geq 0 .
\]
Consequently, the Fourier integral in \eqref{besselintegration} is
\begin{equation}\label{alambdageom}
\int_{-\infty}^\infty \widehat{\chi}(s)\widehat{\tilde{\psi}}(\lambda(|z|-s),\lambda|z|)\lambda ds,
\end{equation}
which is seen to be smooth in $|z|$ with derivatives which are $O(\lambda^{-\frac 12})$ and is $O(\lambda^{-N})$ for any $N>0$ when $|z| \notin (\delta,2\delta)$ (since $\supp(\widehat{\chi}) \subset (\delta,2\delta)$).  This establishes \eqref{convokernel} and in turn \eqref{specpbds}.

\subsubsection{The geometric contribution on a flat cone}
Let us recall from \eqref{geowk} that the ``geometric" contribution to $\frac{\sin(t\sqrt{ \smash[b]{ \Delta_g } })}{\sqrt{ \smash[b]{ \Delta_g } }}$ on $C(\mathbb{S}^1_\rho)$ when $\rho >1$ has a Schwartz kernel of the form (neglecting harmless constants as before)
\[
\left( t^2-G^2 (r_1,r_2;\theta_1-\theta_2)\right)^{-\frac 12}_+, \text{ where } G(r_1,r_2;\theta) \defeq (r_1^2+r_2^2-2r_1r_2\cos\theta)^{\frac 12},
\]
and supported in $|(\theta_2-\theta_1)\mod{ 2 \pi \rho}| \leq \pi$.  Consequently, given two points $(r_1,\theta_1)$, $(r_2,\theta_2)$ such that $|\theta_1-\theta_2|\leq \pi\rho$, the previous subsection shows that the leading order contribution of this term in \eqref{evenintegral} gives rise to the (real part of the) kernel
\[
K(r_1,r_2;\theta_1-\theta_2) \defeq \mathbf{1}_{[-\pi,\pi]}(\theta_1-\theta_2)
e^{i\lambda G(r_1,r_2;\theta_1-\theta_2)}a_\lambda(G(r_1,r_2;\theta_1-\theta_2)),
\]
and we recall that $\supp(a_\lambda)\subset (\delta,2\delta)$.  Note that the factor of $\lambda^{\frac12}$ is not included here and we will thus show that this integral operator contributes to a gain of $\lambda^{-\frac 2p}$ in the $L^2 \to L^p$ estimates for $p=6,\infty$.  For
\[
\supp( g) \subset \{(r,\theta)\in C(\mathbb{S}^1_\rho): r \in (0,4\delta)\}, \]
we have the straightforward $L^\infty $ bound
\[
\sup_{(r_2,\theta_2)}\left| \int K(r_1,r_2;\theta_2-\theta_1) g(r_1,\theta_1)r_1dr_1d\theta_1\right| \lesssim \|g\|_{L^2(r_1dr_1d\theta_1)}.
\]
Note that without loss of generality, we can always assume localization of $g$ this throughout the proof of our theorem on cones due to the localization of $\hat \chi$.

Consequently, we are left to show $L^2 \to L^6$ bounds on the operator determined by $K$.  Due to the sharp cutoff to $|\theta| < \pi$, there is a jump to contend with and the estimates are not a trivial consequence of the standard theory.  It suffices to further assume that $g$ is supported in a small arc of length $\veps$ where $\veps$ is sufficiently small, but otherwise uniform.  In particular, we assume that $\veps<\min(\pi(\rho-1)/2, \pi/4)$.  We then take coordinates such that
\[
\supp(g) \subset \{ (r_1,\theta_1) \in C(\mathbb{S}^1_\rho): r_1 \in (0,4\delta), \theta_1 \in (0,\veps) \}.
\]
Taking coordinates $(r_2,\theta_2)$ such that $\theta_2 \in (-\pi\rho,\pi\rho]$, we now have that
\begin{multline*}
\left\| \mathbf{1}_{(-\pi\rho,\pi]}(\theta_2)\int K(r_1,r_2;\theta_2-\theta_1) g(r_1,\theta_1)r_1dr_1d\theta_1\right\|_{L^6(r_2 dr_2 d\theta_2 )}\\
\lesssim \lambda^{-\frac 13} \|g\|_{L^2(r_1dr_1d\theta_1)}.
\end{multline*}
Indeed, given our assumptions on $\supp(g)$, $K$ vanishes for $\theta_2 < -\pi$.  Hence the characteristic function $\mathbf{1}_{(-\pi\rho,\pi]}(\theta_2)$ ensures that the integral operator identifies with the operator determined by \eqref{convokernel} on $\RR^2$, at which point the bound follows from the standard theory of Carleson-Sj\"olin oscillatory integral operators referenced above.

We are left to show that
\begin{multline*}
\left\| \mathbf{1}_{(\pi,\pi+\veps)}(\theta_2)\int K(r_1,r_2;\theta_2-\theta_1) g(r_1,\theta_1)r_1dr_1d\theta_1\right\|_{L^6(r_2 dr_2 d\theta_2 )}
\\
\lesssim \lambda^{-\frac 13} \|g\|_{L^2(r_1dr_1d\theta_1)}.
\end{multline*}
Let $\eta \in C_c^\infty(\RR)$ be such that $\supp(\eta) \subset(\pi-2\veps,\pi+2\veps)$ and $\eta(\theta)=1$ for $\theta \in [\pi-\veps,\pi+\veps]$.  Given the support hypothesis on the data $g$, we may replace $K(r_1,r_2;\theta_2-\theta_1)$ by $K(r_1,r_2;\theta_2-\theta_1) \eta(\theta_2-\theta_1)$.
Moreover, by applying the inequalities of Minkowski and H\"older in the $r_1$ variable, it suffices to show that with $r_1\in (0,4\delta)$ fixed,
\begin{equation}\label{restrictedgeom}
\left\| \int K(r_1,r_2;\theta_2-\theta_1)\eta(\theta_2-\theta_1)f(\theta_1)d\theta_1\right\|_{L^6(r_2 dr_2 d\theta_2 )}\lesssim \lambda^{-\frac 13}r_1^{-\frac 12} \|f\|_{L^2(d\theta_1)}.
\end{equation}

To show \eqref{restrictedgeom}, we let $T^{\lambda,r_1}$ denote the oscillatory integral operator defined by the left hand side of the inequality.  Consider the mapping defined by
\[
(T^{\lambda,r_1}_{r_2}f)(\theta_2) = \int K(r_1,r_2;\theta_2-\theta_1)\eta(\theta_2-\theta_1)f(\theta_1)d\theta_1
\]
so that for a function $F \in L^{\frac 65}(\tilde{r}_2d\tilde{r}_2d\theta_1)$, we have
\[
\left(T^{\lambda,r_1}\circ(T^{\lambda,r_1})^*(F)\right)(r_2,\theta_2) = \int_{-\infty}^\infty \left(T^{\lambda,r_1}_{r_2}\circ(T^{\lambda,r_1}_{\tilde{r}_2})^*(F(\tilde{r}_2,\cdot))\right)(\theta_2) \tilde{r}_2d\tilde{r}_2.
\]
A standard duality argument implies that \eqref{restrictedgeom} will follow from
\[
\left\| \left(T^{\lambda,r_1}\circ(T^{\lambda,r_1})^*(F)\right) \right\|_{L^6(r_2dr_2d\theta_2)} \lesssim \lambda^{-\frac 23} r_1^{-1} \|F\|_{L^{\frac 65}(\tilde{r}_2d\tilde{r}_2d\theta_1)}
\]
which in turn follows from interpolating the bounds
\begin{align}
\left\|\left(T^{\lambda,r_1}_{r_2}\circ(T^{\lambda,r_1}_{\tilde{r}_2})^*(f)\right) \right\|_{L^\infty(d\theta_2)} &\lesssim (\lambda|r_2-\tilde{r}_2|)^{-\frac 12} r_1^{-1} \|f\|_{L^1(d\theta_1)}
\label{l1linfbound}
\\
\left\|\left(T^{\lambda,r_1}_{r_2}\circ(T^{\lambda,r_1}_{\tilde{r}_2})^*(f)\right) \right\|_{L^2(d\theta_2)} &\lesssim (\lambda r_1r_2)^{-\frac 12}(\lambda r_1\tilde{r}_2)^{-\frac 12}\|f\|_{L^2(d\theta_1)}\label{l2l2bound}
\end{align}
followed by the Hardy-Littlewood-Sobolev fractional integration inequality in $r_2$.

To see \eqref{l1linfbound}, we first observe that the Schwartz kernel of $T^{\lambda,r_1}_{r_2}\circ(T^{\lambda,r_1}_{\tilde{r}_2})^*$ is
\begin{equation}\label{ttstarkernel}
\int \mathbf{1}_{[\pi-2\veps,\pi]}(\theta_1-\theta)\mathbf{1}_{[\pi-2\veps,\pi]}(\theta_2-\theta)e^{i\lambda \Psi_{r_1}(r_2,\tilde{r}_2,\theta_1,\theta_2,\theta)} A_{\lambda,r_1}(r_2,\tilde{r}_2,\theta_1,\theta_2,\theta)\,d\theta
\end{equation}
where
\[
\Psi_{r_1}(r_2,\tilde{r}_2,\theta_1,\theta_2,\theta) =G(r_1,r_2,\theta_2-\theta)- G(r_1,\tilde{r}_2,\theta_1-\theta)
\]
and
\[
A_{\lambda,r_1}(r_2,\tilde{r}_2,\theta_1,\theta_2,\theta) = a_\lambda(G(r_1,r_2,\theta_2-\theta))\overline{a_\lambda}(G(r_1,\tilde{r}_2,\theta_1-\theta))
\eta(\theta_2-\theta)\eta(\theta_1-\theta).
\]
Note that given the compact support of $\eta$, including the left endpoint $\pi-2\veps$ in the characteristic functions in \eqref{ttstarkernel} is redundant, but is kept for emphasis.  The limits of integration in the integral can thus be taken as $\max(\theta_1-\pi, \theta_2-\pi)$ and $\min(\theta_1-\pi+2\veps,\theta_2-\pi+2\veps)$.  The bound \eqref{l1linfbound} will follow by applying  standard oscillatory integral estimates to \eqref{ttstarkernel}.  We may assume that $r_1^2|r_2-\tilde{r_2}|\geq \lambda^{-1}$ below as the other case is trivial.  The crucial lower bound is thus
\begin{equation}\label{lowerbound2nd}
\left|\prtl_{\theta}\Psi_{r_1}(r_2,\tilde{r}_2,\theta_1,\theta_2,\theta) \right| + \left|\prtl_{\theta}^2\Psi_{r_1}(r_2,\tilde{r}_2,\theta_1,\theta_2,\theta) \right| \gtrsim |r_2-\tilde{r}_2|r_1^2.
\end{equation}
Once this is established, either the phase lacks critical points or we can appeal to the stationary phase estimates in \cite[\S VIII.1.2]{steinharmonic}. In either case, we have that \eqref{ttstarkernel} is $O( (\lambda|r_2-\tilde{r}_2|)^{-\frac 12} r_1^{-1} )$.  Note that this version of stationary phase is uniform regardless of the location of any critical points in the domain of integration.

To see the lower bound \eqref{lowerbound2nd}, first observe that by taking $\veps$ sufficiently small, we may assume that if $A_{\lambda,r_1}(r_2,\tilde{r}_2,\theta_1,\theta_2,\theta)\neq 0$, then $(r_2+r_1),(\tilde{r}_2+r_1)  \geq \delta/4$.  Next, by using that
\[
 \frac{1}{G(r_1,r_2,\theta)}= \frac{1}{r_1+r_2}+\frac{r_1r_2(\cos\theta+1)}{(r_1+r_2)^{3}}
\]
we see that
\begin{align}
\prtl_\theta G(r_1,r_2,\theta) &=\frac{r_1r_2}{r_1+r_2}\sin(\theta) + O((r_1r_2)^2(\theta-\pi)^3)\label{Gfirstderiv}\\
\prtl^2_\theta G(r_1,r_2,\theta) &= \frac{r_1r_2}{r_1+r_2}\cos(\theta) + O((r_1r_2)^2(\theta-\pi)^2).\label{Gsecondderiv}
\end{align}
Hence up to acceptable error, $\prtl_{\theta}^2\Psi_{r_1}(r_2,\tilde{r}_2,\theta_1,\theta_2,\theta)$ is
\[
\frac{r_1^2(r_2-\tilde{r}_2)}{(r_1+r_2)(r_1+\tilde{r}_2)}\cos(\theta-\theta_2) +
\frac{r_1\tilde{r}_2}{r_1+\tilde{r}_2} \left(\cos(\theta-\theta_2)-\cos(\theta-\theta_1)\right).
\]
The first term has the proper lower bound we desire and the estimate will hold provided the second term is lower order.  The second term is $O(r_1\tilde{r}_2\veps|\theta_1-\theta_2|)$, so this shows \eqref{lowerbound2nd} provided
\[ r_1\tilde{r}_2|\theta_1-\theta_2| \leq C\veps^{-1} r_1^2|r_2-\tilde{r}_2| \]
for some constant $C$ which can be taken as absolute as long as $\veps$ is sufficiently small.  We are thus left to see that if
\[r_1\tilde{r}_2|\theta_1-\theta_2| > Cr_1^2|r_2-\tilde{r}_2|,\]
then the first derivative term in \eqref{lowerbound2nd} satisfies the lower bound.  But \eqref{Gfirstderiv} shows that up to acceptable error, $\prtl_{\theta}\Psi_{r_1}(r_2,\tilde{r}_2,\theta_1,\theta_2,\theta)$ is
\[
 \frac{r_1^2(r_2-\tilde{r}_2)}{(r_1+r_2)(r_1+\tilde{r}_2)}\sin(\theta-\theta_2) +
\frac{r_1\tilde{r}_2}{r_1+\tilde{r}_2} \left(\sin(\theta-\theta_2)-\sin(\theta-\theta_1)\right).
\]
This time we may take $\veps$ small so that the absolute value of the second term is bounded below by $|\theta_1-\theta_2|(r_1\tilde{r}_2)/(4\delta)$.  Since the first term is $O(\veps r_1^2|r_2-\tilde{r}_2|)$, the inequality now follows by taking $\veps$ sufficiently small.

The bound \eqref{l2l2bound} will follow from
\[
\|T^{\lambda,r_1}_{r_2}\|_{L^2(d\theta_1)\to L^2(d\theta_2)} \lesssim (\lambda r_1 r_2)^{-\frac 12}.
\]
The assumption on the data $g$ and the supports of $K$, $\eta$ in $\theta_2-\theta_1$ mean that we may treat the $\theta_i$ as variables in $\RR$.  Since $T^{\lambda,r_1}_{r_2}$ is a convolution kernel in $\theta$ it thus suffices to show that the corresponding Fourier multiplier satisfies
\[
\left|\int_{\pi-2\veps}^{\pi} e^{i\zeta + i\lambda G(r_1,r_2,\theta)} a_\lambda\left(G(r_1,r_2,\theta)\right)
\eta(\theta)d\theta\right| \lesssim (\lambda r_1 r_2)^{-\frac 12}.
\]
But this follows from \eqref{Gsecondderiv} and the same version of stationary phase used above, once we recall that we may assume that $r_1+r_2 \geq \delta/4$.

\subsection{The diffractive contribution}
Recall from \eqref{diffwk}, the contribution of the diffractive term is supported where $t>r_1+r_2$ and is a sum of terms of the form
\[
-\frac{1}{4\pi^2\rho}\int_0^{\cosh^{-1}\left(\frac{t^2-r_1^2-r_2^2}{2r_1r_2}\right)}
\left( t^2-r_1^2-r_2^2-2r_1r_2\cosh(s)\right)^{-\frac 12}\frac{\sin\varphi}{\cosh(\frac s\rho)-\cos\varphi}\,ds
\]
where $\varphi = (\pi\pm(\theta_1-\theta_2))/\rho$.  Recalling the left hand side of \eqref{evenintegral} and that $\supp(\widehat{\chi})\subset(\delta,2\delta)$, this has a nontrivial contribution to the Schwartz kernel only when $r_1+r_2 <2\delta$.  In this case, reasoning as in \eqref{fourierfund} the contribution is
\begin{multline}\label{diffcontrib}
-\frac{\lambda}{4\pi^2\rho}\int_{-\infty}^\infty \int_{r_1+r_2}^\infty \int_0^{\cosh^{-1}\left(\frac{t^2-r_1^2-r_2^2}{2r_1r_2}\right)} \chi(\tau) \left( \sin(t(\lambda+\tau))+ \sin(t(\lambda-\tau))\right)\\
\times\left( t^2-r_1^2-r_2^2-2r_1r_2\cosh(s)\right)^{-\frac 12}\frac{\sin\varphi}{\cosh(\frac s\rho)-\cos\varphi}\,ds
dt d\tau=\\
-\frac{\lambda}{2\pi^2\rho}\int_{-\infty}^\infty \int_{r_1+r_2}^\infty \int_0^{\cosh^{-1}\left(\frac{t^2-r_1^2-r_2^2}{2r_1r_2}\right)} \chi(\tau) \sin(t(\lambda-\tau))\\
\times\left( t^2-r_1^2-r_2^2-2r_1r_2\cosh(s)\right)^{-\frac 12}\frac{\sin\varphi}{\cosh(\frac s\rho)-\cos\varphi}\,ds
dt d\tau
.
\end{multline}
where the second expression follows from the fact that $\tau$ is even. Switching the order of integration in $t,s$ yields
\begin{multline*}
-\frac{\lambda}{2\pi^2\rho}\int_{-\infty }^\infty \chi(\tau) \int_0^\infty\frac{\sin\varphi}{\cosh(\frac s\rho)-\cos\varphi}\times\\
\left(\int_{(r_1^2+r_2^2+2r_1r_2\cosh s)^{\frac 12} }^\infty \frac{\sin(t(\lambda-\tau))}{\left( t^2-r_1^2-r_2^2-2r_1r_2\cosh s \right)^{\frac 12}}\,dt\right)
ds d\tau.
\end{multline*}
Neglecting harmless constants, \eqref{besselidentity} shows that this is
\[
\lambda\int_{-\infty }^\infty \chi(\tau) \int_0^\infty\frac{\sgn(\lambda-\tau)\sin\varphi}{\cosh(\frac s\rho)-\cos\varphi}J_0\left((r_1^2+r_2^2+2r_1r_2\cosh s)^{\frac 12}|\lambda-\tau|\right)\,
ds d\tau.
\]
Proceeding as above, we use the smooth, even bump function $\psi$ satisfying $\supp(\psi) \subset (-\frac 12,\frac 12)$ and $\supp(1-\psi) \subset (-\frac 14, \frac 14)^c$.  As before, it suffices to restrict attention to the integral
\begin{equation}\label{diffhankelintegral}
\lambda\int_{-\infty}^\infty \chi(\tau)\psi(\tau/\lambda) \int_0^\infty\frac{\sin\varphi}{\cosh(\frac s\rho)-\cos\varphi}J_0\left((r_1^2+r_2^2+2r_1r_2\cosh s)^{\frac 12}(\lambda-\tau)\right)\,
ds d\tau
\end{equation}
as the error is $O(\lambda^{-N})$ uniformly in $\varphi$ by the same argument as above.  We now follow the same approach as in \eqref{alambdageom} and in particular we use the function $\tilde{\psi}$ from that same discussion. First set
\[
D(r_1,r_2,s) \defeq (r_1^2+r_2^2+2r_1r_2\cosh s)^{\frac 12}.
\]
Similarly to \eqref{alambdageom}, we define
\begin{equation}\label{alambdadiff}
a_\lambda(\zeta) = \int\widehat{\chi}(s)\widehat{\tilde{\psi}}(\lambda(\zeta-s);\lambda \zeta)\lambda ds,
\end{equation}
so that $a_\lambda$ is rapidly decreasing outside a $\lambda^{-1}$ neighborhood of $\supp(\chi)$. Note that in contrast to the previous sections, there is a small difference in the definition of $a_\lambda$ here in that we do not assume that it is compactly supported in $(\delta,2\delta)$. As in the geometric case, from \eqref{diffhankelintegral} it now suffices to consider the real part of
\begin{equation}\label{kerneldef}
K(r_1,r_2,\theta) \defeq \int_0^\infty e^{i\lambda D(r_1,r_2,s)}
\frac{\sin(\frac{\theta}{\rho})}{\cosh(\frac s\rho)-\cos (\frac{\theta}{\rho})}a_\lambda\left(D(r_1,r_2,s)\right) \,ds
\end{equation}
as we may take translations in $\theta$ to remove the term $\pi$ in the definition of $\varphi$.  As before, we remove the factor of $\lambda^{\frac 12}$ in the kernel and instead we will prove the following bound.

\begin{lemma}
\label{lemkerbd}
For $p=6$ and $p=\infty$ and $K(r_1,r_2,\theta_2-\theta_1) $ defined in \eqref{kerneldef}, we obtain the gain
\begin{equation}\label{diffLpL2}
\left\|\iint K(r_1,r_2,\theta_2-\theta_1) g(r_1,\theta_1)r_1dr_1 d\theta_1\right\|_{L^p(r_2d\theta_2dr_2)} \lesssim \lambda^{-\frac 2p} \left\|g\right\|_{L^2(r_1d\theta_1dr_1)}.
\end{equation}
\end{lemma}

We begin by observing that if $ s>0$, $2 < p \leq \infty$, and $q>1$ satisfies $\frac 1q = \frac 12+ \frac 1p$ then
\begin{equation}\label{thetayoungs}
\left( \int_{\mathbb{S}^1_\rho} \left|\frac{\sin(\frac{\theta}{\rho})}{\cosh(\frac s\rho)-\cos(\frac{\theta}{\rho})}\right|^{q}\,d\theta \right)^{\frac 1q}
\lesssim \begin{cases}
s^{\frac 1q-1} & s <1\\
e^{-\frac{s}{\rho }} & s \geq 1
\end{cases}
\end{equation}
Note that as a consequence for each $s$, the integrand in \eqref{kerneldef} is in $L^q(r_i dr_i d\theta )$ when $i=1,2$ and hence by  Young's inequality this maps $L^2\to L^p$ with operator norm integrable in $s$.  In particular, by Minkowski's integral inequality this already yields \eqref{diffLpL2} when $p=\infty$.  For the remainder of the section, we thus assume that $p=6$.  Moreover, since $s^{-\frac 13}$ presents a locally integrable singularity, this gives the estimate
\begin{multline}\label{smallr2}
\left\|\mathbf{1}_{(0,\frac{1}{\lambda})}(r_2)\iint K(r_1,r_2,\theta_2-\theta_1) g(r_1,\theta_1)r_1\,dr_1 d\theta_1\right\|_{L^6(r_2d\theta_2dr_2)}\\ \lesssim \left( \int_0^{\lambda^{-1}} \left| \int_0^{2\delta}\|g(r_1,\cdot)\|_{L^2(d\theta_1)}r_1dr_1 \right|^6\,r_2dr_2\right)^{\frac 16} \lesssim \lambda^{-\frac 13} \left\|g\right\|_{L^2(r_1d\theta_1dr_1)},
\end{multline}
showing that we may further localize $K$ to $r_2\geq \lambda^{-1}$.  The same argument shows that the same restricted kernel maps $L^{\frac 65}\to L^2$ with an even stronger gain.  Consequently, by duality and symmetry of $K$ in $r_1,r_2$ we may also restrict attention to $r_1 \geq \lambda^{-1}$.

\subsubsection{The case $r_1+r_2 \approx \delta$}  Here we obtain estimates on the contribution of
\[
\mathbf{1}_{(\lambda^{-1},\infty)}(r_1)\mathbf{1}_{(\lambda^{-1},\infty)}(r_2)
\mathbf{1}_{(\frac{\delta}{2},2\delta)}(r_1+r_2)K(r_1,r_2,\theta_2-\theta_1)
\]
to \eqref{diffLpL2}. Here we also note that it suffices to replace $K$ by $e^{-i\lambda(r_1+r_2)}K$ in \eqref{diffLpL2} as these modulations to $g$ and its image under the integral operator do not change their $L^p$ norms.  Define the following integral kernel $H$ which will serve as a sufficiently accurate approximation to  $e^{-i\lambda(r_1+r_2)}K(r_1,r_2,\theta_2-\theta_1)$
\begin{equation}\label{Hdef}
H(r_1,r_2,\theta) \defeq 2\rho(r_1+r_2)^{\frac 12}a_\lambda(r_1 + r_2) \int_0^\infty e^{i\lambda r_1 r_2s^2} \frac{\theta}{s^2+\theta^2}\,ds.
\end{equation}
Strictly speaking, we should be multiplying $H$ by characteristic functions which localize us to $r_1, r_2 \geq \lambda^{-1}$ and $r_1+r_2 \in (\delta/2,2\delta)$ but we suppress these for both $H$ and $K$ to avoid cluttering the notation.

\begin{lemma} Suppose $r_1, r_2 \geq \lambda^{-1}$ and $r_1+r_2 \in (\delta/2,2\delta)$ as above.  Then the difference
\[
\tilde{K}(r_1,r_2,\theta) =e^{-i\lambda(r_1+r_2)} K(r_1,r_2,\theta) -H(r_1,r_2,\theta)
\]
satisfies
\begin{equation}\label{ktildebound}
|\tilde{K}(r_1,r_2,\theta) | \lesssim (\lambda r_1 r_2)^{-\frac 12}.
\end{equation}
\end{lemma}
Given the lemma, by applying Young's inequality in $\theta$ along with the inequalities of Minkowski and H\"older, we have
\begin{multline}\label{spcontribution}
\left\|\iint \tilde{K}(r_1,r_2,\theta_2-\theta_1) g(r_1,\theta_1)r_1dr_1 d\theta_1\right\|_{L^6(r_2d\theta_2dr_2)}\\
\lesssim \lambda^{-\frac12} \left(\int_{\lambda^{-1}}^1 \left|\int_{\lambda^{-1}}^1 \|g(r_1,\cdot)\|_{L^2(d\theta_1)}r_1^{\frac 12}dr_1\right|^6r_2^{-2}dr_2 \right)^{\frac 16} \lesssim \lambda^{-\frac 13}  \|g\|_{L^2(r_1dr_1d\theta_1)}.
\end{multline}
Hence it suffices to replace $e^{-i\lambda(r_1+r_2)}K$ by $H$ below.

The lemma makes use of the stationary phase estimates in \cite[\S VIII.1.2]{steinharmonic} or \cite[\S2.9]{Erdelyi} which imply if $a\in C^{1}([0,1])$ and $\phi(s)$ has a single nondegenerate critical point at $s=0$, then
\begin{equation}\label{statphase}
\int_0^1 e^{i\mu\phi(s)} a(s)\,ds = O(\mu^{-\frac 12}).
\end{equation}

\begin{proof}
We will see that $\tilde{K}$ is a sum of terms, each of which is $O((\lambda r_1 r_2)^{-\frac 12} )$. Note that while the domain of integration in \eqref{kerneldef} is over $[0,\infty)$, we can include the contribution of the integral over $[1,\infty)$ in $\tilde{K}$ provided $\delta$ is sufficiently small as the phase function over this interval lacks any critical points.

We first observe that the difference between the integral defining $K$, now restricted to $s\in [0,1]$, and the integral
\begin{equation}\label{firstapprox}
\int_0^1 e^{i\lambda(r_1^2+r_2^2+2r_1r_2\cosh s)^{\frac 12}-i\lambda(r_1+r_2)}
\frac{2\rho\theta}{s^2+\theta^2} a_\lambda(r_1 + r_2)\,ds
\end{equation}
can be included in $\tilde{K}$.  Observe that for $s \in [0,1]$
\[
\left|\prtl_s^j\left(\frac{\sin(\frac{\theta}{\rho})}{\cosh(\frac s\rho) -\cos(\frac{\theta}{\rho}) } - \frac{2\rho\theta}{s^2+\theta^2} \right)\right| \lesssim \theta^{1-j}, \qquad j=0,1.
\]
Since \eqref{statphase} only requires that the amplitude is $C^1$, replacing $\frac{\sin(\frac{\theta}{\rho})}{\cosh(\frac s\rho) -\cos(\frac{\theta}{\rho}) }$ by this difference in \eqref{kerneldef} yields a term which is $O((\lambda r_1 r_2)^{-\frac 12} )$ and hence this can indeed be included in $\tilde{K}$.  Moreover,  a Taylor expansion shows that
\begin{equation}\label{phaseexp}
\left( r_1^2+r_2^2 + 2r_1r_2\cosh s \right)^{\frac 12}-(r_1+r_2) = \frac{r_1r_2 s^2}{r_1+r_2}\left(1+ O(r_1r_2s^2)\right).
\end{equation}
Hence
\[
\left|\prtl_s^j\left( a_\lambda\left((r_1^2+r_2^2+2r_1r_2\cosh s)^{\frac 12}\right)-a_\lambda(r_1 + r_2) \right)\right| \lesssim s^{2-j},\qquad j=0,1,
\]
and applying \eqref{statphase} a second time completes the proof of the claim that error introduced by replacing the integral in $K$ by \eqref{firstapprox} is acceptable.

We now make a change of variables in \eqref{firstapprox}, defining $\tilde{s}$ as a function of $s$ by
\[
r_1r_2 \tilde{s}^2 = \left( r_1^2+r_2^2 + 2r_1r_2\cosh s \right)^{\frac 12}-(r_1+r_2).
\]
Therefore by \eqref{phaseexp}, we have
$
\frac{d\tilde{s}}{ds} = \frac{1}{(r_1+r_2)^{\frac 12}}+O(r_1r_2s^2)
$
and hence since $s=O(\tilde{s})$,
\[
\frac{ds}{d\tilde{s}} = (r_1+r_2)^{\frac 12}+O(r_1r_2\tilde{s}^2)
\]
Applying stationary phase as before, the contribution of the second term on the right here can be included into $\tilde{K}$.  If the domain of integration of the integral in \eqref{Hdef} were over $[0,1]$, this would conclude the proof.  Since this is not the case, we simply observe that
\[
\int_1^\infty e^{i\lambda r_1 r_2s^2} \frac{\theta}{s^2+\theta^2}\,ds
=
O\left((\lambda r_1r_2)^{-1}\right)
\]
as the phase function here lacks critical points.
\end{proof}

Returning to $H$ in \eqref{Hdef}, we redefine $\psi$ to be an even bump function supported in a small interval about the origin of size much less than $\rho$.  Define $\widetilde{H}$ as the kernel obtained by multiplying the expression on the left hand side of \eqref{Hdef} by $\psi(\theta)$. Applying stationary phase, we have that
\[
2\rho(1-\psi)(\theta)(r_1+r_2)^{\frac 12}a_\lambda(r_1 + r_2)
\int_0^\infty e^{i\lambda r_1 r_2s^2} \frac{\theta}{s^2+\theta^2}\,ds = O\left((\lambda r_1 r_2)^{-\frac 12}\right).
\]
Given \eqref{spcontribution}, it now suffices to show \eqref{diffLpL2} with $K$ replaced by $\widetilde{H}$.

Next, we recall that the convolution kernel defined for $\theta \in \RR$ by
\[
Q_s(\theta) \defeq \pi \frac{\theta}{s^2+\theta^2}
\]
is known as the \emph{conjugate Poisson kernel}, see for example \cite[p.265]{Grafakos}.  Its action is equivalent to the Fourier multiplier with symbol
\[
\widehat{Q_s}(\xi) = -i\sgn(\xi)e^{-s|\xi|}
\]
and hence if we let $H(r_1,r_2,\theta)$ be the expression in \eqref{Hdef} but with $\theta \in \RR$ (as opposed to $\theta \in \mathbb{S}^1_\rho$), we have that its partial Fourier transform in $\theta$ satisfies
\begin{align}
\widehat{H}(r_1,r_2,\xi) &= \frac{2\rho}{i\pi}\sgn(\xi)(r_1+r_2)^{\frac 12}a_\lambda(r_1 + r_2) \int_0^\infty e^{i\lambda r_1 r_2 s^2}
e^{-s|\xi|}\,ds \notag\\
&= \frac{2\rho}{i\pi}(r_1+r_2)^{\frac 12}\sgn(\xi)a_\lambda(r_1 + r_2)
\left(\sum_{k=0}^\infty \frac{(-1)^k|\xi|^k}{k!} \int_0^\infty e^{i\lambda r_1 r_2 s^2} s^k\,ds\right). \label{Hhat}
\end{align}

\begin{lemma}\label{lem:fouriermultbd}
The multipliers $\widehat{H}(r_1,r_2,\xi)$, $\widehat{\widetilde{H}}(r_1,r_2,\xi)$ satisfy the bounds
\begin{equation}\label{fouriermultbd}
|\widehat{\widetilde{H}}(r_1,r_2,\xi)|, |\widehat{H}(r_1,r_2,\xi) | \lesssim
\begin{cases} (\lambda r_1 r_2)^{-\frac 12} & |\xi| \lesssim ( \lambda r_1 r_2)^{\frac{1}{2}} \\
\frac{1}{|\xi|} & |\xi| \gg ( \lambda r_1 r_2 )^{\frac{1}{2}}
\end{cases}.
\end{equation}
\end{lemma}
\begin{proof}  First note that since $\widehat{\psi}$ is a Schwartz class function rapidly decreasing on the unit scale, the bound on $\widehat{\widetilde{H}}=\widehat{\psi} * \widehat{H}$ follows from the one on $\widehat{H}$.  We now consider the identity (cf. \cite[p. 54]{Erdelyi})
\[
\int_0^\infty e^{i \lambda r_1 r_2 s^2} s^k\,ds = \frac 12 \Gamma\left(\frac{k+1}{2}\right)e^{\frac{i\pi}{4}(k+1)}(\lambda r_1 r_2 )^{- \frac{k+1}{2}}.
\]
The power series in parentheses in \eqref{Hhat} thus takes the form
\begin{equation}\label{khat}
\frac 12 \sum_{k=0}^\infty \frac{(-1)^k|\xi|^k}{k!}\Gamma\left(\frac{k+1}{2}\right)e^{\frac{i\pi}{4}(k+1)}
(\lambda r_1 r_2 )^{- \frac{k+1}{2}}.
\end{equation}
The power series
\[
F(z) = \sum_{k=0}^\infty \frac{(-1)^k \Gamma\left(\frac{k+1}{2}\right)}{k!}z^k = \sum_{l=0}^\infty \frac{ \Gamma\left(l+\frac 12 \right)}{(2l)!}z^{2l} +\sum_{l=1}^\infty \frac{ \Gamma\left(l \right)}{(2l-1)!}z^{2l-1},
\]
is seen to converge uniformly on compact sets and satisfies
\[
\widehat{H}(r_1,r_2,\xi) =2\rho(r_1+r_2)^{\frac 12}\sgn(\xi)a_\lambda(r_1 + r_2)  ( \lambda r_1 r_2 )^{- \frac{1}{2}} e^{-\frac{i\pi}{4}}
F\left((\lambda r_1 r_2 )^{- \frac{1}{2}}e^{\frac{i\pi}{4}}|\xi|\right).
\]
The desired bound \eqref{fouriermultbd} will then follow from
\begin{equation}\label{powerseriesbd}
|F(z)| \lesssim \begin{cases} 1 & |z| \leq 1\\
\frac{1}{|z|} & |z| \geq  1
\end{cases},
\end{equation}
for $z \in \mathbb{C}$ such that $\arg(z)=\frac \pi4$. The bound for $|z| \leq 1$ is immediate from the uniform convergence noted above.

To analyze the  behavior of $F(z)$ when $|z|$ is large, we split the series into even and odd terms as above.  When $k=2l$ is even,  the duplication formula 
\[ \Gamma(z)\Gamma(z+\frac 12) = 2^{1-2z}\sqrt{\pi}\Gamma(2z)\] 
gives
\[
\Gamma\left(\frac{k+1}{2}\right) = \Gamma\left(l+\frac{1}{2}\right) = \frac{(2l)!\sqrt{\pi}}{4^l l!}.
\]
The contribution of the even terms to $F(z)$ is therefore easy to describe as
\begin{equation}\label{evencontrib}
\sum_{l=0}^\infty \frac{\sqrt{\pi}}{4^l l!}z^{2l}=\sqrt{\pi}\exp(z^2/4).
\end{equation}
When $k=2l-1$ is odd, we use Pochhammer notation $(\frac 12)_{l} = (\frac 12)(\frac 32)\cdots (l-\frac 12)$ to write
$ (2l-1)! = 2^{2l-1}\left(\frac 12\right)_{l}(l-1)!$. Hence
\[
 \frac{\Gamma\left(\frac{k+1}{2}\right)}{k!} = \frac{(l-1)!}{(2l-1)!} = \frac{2}{4^l(\frac 12)_l}=\frac{2 l!}{4^l(\frac 12)_l l!}.
\]
The odd terms thus yield a series that can be expressed in terms of a Kummer function $\Phi(\alpha,\gamma;w) = \sum_{l=0}^\infty \frac{(\alpha)_l}{(\gamma)_l l!}w^l$ (cf. \cite[\S9.9]{Leb}, though note that several other authors denote this as $M(\alpha,\gamma;w)$):
\[
-\frac{2}{z}\sum_{l=1}^\infty \frac{(1)_l}{(\frac 12)_l l!}\left( \frac{z^2}{4}\right)^l = -\frac{2}{z}\left( \Phi\left(1,\frac 12;\frac{z^2}{4}\right)-1\right).
\]
Using the asymptotics of $\Phi(\alpha,\gamma;w)$ for large $|w|$ and $\arg w =\frac{\pi}{2}$ from \cite[(9.12.7)]{Leb}, we have
\[
\Phi\left(1,\frac 12;w\right) = \Gamma\left(\frac 12\right)\left( \frac{e^w w^{\frac 12}}{\Gamma(1)} +  O(w^{-\frac 12}) \right).
\]
Hence since $\Gamma(\frac 12) = \sqrt{\pi}$:
\[
-\frac{2}{z}\left( \Phi\left(1,\frac 12;\frac{z^2}{4}\right)-1\right) = -\sqrt{\pi}  \exp(z^2/4) + \frac 2z + O(z^{-2}).
\]
The bound \eqref{powerseriesbd} now follows from the cancellation between \eqref{evencontrib} and the first term in the asymptotic expansion here.
\end{proof}

\begin{theorem} The operator determined by the integral kernel $\widetilde{H}$ maps $L^2 \to L^6$ with operator norm bounded by $\lambda^{-\frac 13}$:
\begin{equation}\label{diffL2L6bound}
\left\|\iint \widetilde{H}(r_1,r_2,\theta_2-\theta_1) g(r_1,\theta_1)\,d\theta_1r_1dr_1 \right\|_{L^6(d\theta_2r_2dr_2)} \lesssim \lambda^{-\frac 13}\|g\|_{L^2(d\theta_1r_1dr_1)} .
\end{equation}
\end{theorem}
\begin{proof}

By Minkowski's inequality, \eqref{diffL2L6bound} is reduced to
\begin{equation*}
\int \left\| \widetilde{H}(r_1,r_2,\cdot)* g(r_1,\cdot)\right\|_{L^6(d\theta_2r_2dr_2)} r_1 dr_1 \lesssim \lambda^{-\frac 13}\|g\|_{L^2(d\theta_1r_1dr_1)} .
\end{equation*}
In particular, if we can show that for $f \in L^2(d\theta_1)$ and $r_1 \in (0,\delta)$ fixed, we have
\begin{equation}\label{restrictedr1bd}
 \left\| \widetilde{H}(r_1,r_2,\cdot)* f\right\|_{L^6(d\theta_2r_2dr_2)} \lesssim \lambda^{-\frac 13}r_1^{-\frac 13}  \|f \|_{L^2(d\theta_1)},
\end{equation}
then by H\"older's inequality, \eqref{diffL2L6bound} will follow.

Now let $\{\beta_\ell\}_{\ell=0}^\infty$ be a sequence of smooth Littlewood-Paley cutoffs satisfying for $\xi \in \RR$
\[
\sum_{\ell=0}^\infty \beta_\ell (\xi) = 1 \text{ and } \beta_\ell (\xi) = \beta(2^{1-\ell}\xi) \text{ for } \ell \geq 1
\]
with $\supp(\beta) \subset \{|\xi| \in (\frac 12,2) \}$ and $\supp(\beta_0) \subset (-2,2)$. Now define $\widetilde{H}_\ell$ by
\[
\widehat{\widetilde{H}}_\ell (r_1,r_2,\xi) \defeq \beta_\ell (\xi) \widehat{\widetilde{H}}(r_1,r_2,\xi).
\]
By Sobolev embedding/Young's inequality we have that
\begin{equation}\label{sobembed}
\|H_\ell (r_1,r_2,\cdot)*f\|_{L^6(d\theta_2)}  \lesssim
\begin{cases}
(\lambda r_1 r_2)^{-\frac 12}2^{\frac{\ell}{3}}\|\beta_\ell f\|_{L^2(d\theta_1)}  & 2^\ell \leq (\lambda r_1 r_2)^{\frac 12}\\
2^{-\frac {2 \ell}3}\|\beta_\ell f\|_{L^2(d\theta_1)}  & 2^\ell > (\lambda r_1 r_2)^{\frac 12}
\end{cases}.
\end{equation}

The classical Littlewood-Paley square function bound and Minkowski's inequality imply that the left hand side of \eqref{restrictedr1bd} is dominated by
\begin{equation*}
\left( \int \left( \sum_{\ell=0}^\infty \| H_\ell (r_1,r_2,\cdot)*f\|_{L^6(d\theta_2)}^2\right)^3r_2dr_2\right)^{\frac 16} .
\end{equation*}
Applying \eqref{sobembed}, this in turn is bounded by
\begin{multline*}
(\lambda r_1)^{-\frac 12} \left( \int_{\lambda^{-1}}^1 \left(\sum_{2^\ell \leq (\lambda r_1r_2)^{\frac 12}} 2^{\frac{2 \ell }3}
\|\beta_\ell f \|_{L^2(d\theta_1)}^2\right)^3\,r_2^{-2}dr_2\right)^{\frac 13\cdot \frac 12}
\\
+
\left( \int_{\lambda^{-1}}^1 \left(\sum_{2^\ell > (\lambda r_1r_2)^{\frac 12}} 2^{-\frac{4\ell}3}
\|\beta_\ell f \|_{L^2(d\theta_1)}^2\right)^3\,r_2dr_2\right)^{\frac 13\cdot \frac 12} .
\end{multline*}
To bound the first expression here we use Minkowski's inequality to get that
\begin{multline*}
(\lambda r_1)^{-1} \left( \int \left(\sum_{2^\ell \leq (\lambda r_1r_2)^{\frac 12}} 2^{\frac{2 \ell}3}
\|\beta_\ell f\|_{L^2(d\theta_1)}^2r_2^{-\frac 23}\right)^3\,dr_2\right)^{\frac 13}
\\
\lesssim
(\lambda r_1)^{-1} \sum_{\ell=0 }^\infty 2^{\frac{2 \ell }3}\|\beta_\ell f\|_{L^2(d\theta_1)}^2\left( \int_{2^{2 \ell}(\lambda r_1)^{-1} \leq r_2} r_2^{-2}\,dr_2\right)^{\frac 13} \\
\lesssim (\lambda r_1)^{-\frac 23}\sum_{\ell=0 }^\infty  \|\beta_\ell f\|_{L^2(d\theta_1)}^2 \approx
(\lambda r_1)^{-\frac 23}\|f\|_{L^2(d\theta_1)}^2.
\end{multline*}
and after taking square roots, the contribution of this expression is bounded above by the right hand side of \eqref{restrictedr1bd}.  For the second expression, we use Minkowski's inequality again
\begin{multline*}
\left( \int_{\lambda^{-1}}^1 \left(\sum_{2^\ell > (\lambda r_1r_2)^{\frac 12}} 2^{-\frac{4\ell}3}
\|\beta_\ell f\|_{L^2(d\theta_1)}^2r_2^{\frac 13}\right)^3\,dr_2\right)^{\frac 13}
\\
\lesssim
\sum_{\mu } 2^{-\frac{4 \ell}3}\|\beta_\ell f\|_{L^2(d\theta_1)}^2\left( \int_{r_2\leq 2^{2 \ell }(\lambda r_1)^{-1} }r_2\,dr_2\right)^{\frac 13} \\
\lesssim (\lambda r_1)^{-\frac 23}\sum_{\ell=0 }^\infty \|\beta_\ell f\|_{L^2(d\theta_1)}^2 \approx
(\lambda r_1)^{-\frac 23}\|f\|_{L^2(d\theta_1)}^2.
\end{multline*}
The desired bound \eqref{restrictedr1bd} now follows as before after taking square roots.
\end{proof}

\subsubsection{The case $r_1+r_2 \leq \delta/2$}
We now consider the contribution of
\[
\mathbf{1}_{(\lambda^{-1},\infty)}(r_1)\mathbf{1}_{(\lambda^{-1},\infty)}(r_2)
\mathbf{1}_{(0,\frac{\delta}{2})}(r_1+r_2)K(r_1,r_2,\theta_2-\theta_1)
\]
to \eqref{diffLpL2}.  As before, we will assume that $r_1,r_2$ evaluate to one along the characteristic functions here throughout this subsection to avoid cluttering notation. Recall that the amplitude $a_\lambda$ in \eqref{alambdadiff} decays rapidly outside a $\lambda^{-1}$ neighborhood of $\supp(\widehat{\chi})\subset (\delta, 2\delta)$.  Hence for any $N>0$,
$ a_\lambda(D(r_1,r_2,s)) = O(\lambda^{-N}) $ if $D(r_1,r_2,s) \notin (\delta,2\delta)$.  But $D(r_1,r_2,s) \in (\delta,2\delta)$ implies that
\[
2 r_1r_2(\cosh s-1) \geq \delta^2 - (r_1+r_2)^2 \geq \frac{3\delta^2}{4} ,
\]
which means there exists a sufficiently small constant $c_0$ such that $ s \geq c_0(r_1r_2)^{-\frac 12}$.  Since we are assuming $r_i\geq \lambda^{-1}$, $i=1,2$, we can recall \eqref{thetayoungs} to see that the integral operator with kernel
\[
\int_0^{c_0(r_1r_2)^{-\frac 12}} e^{i\lambda D(r_1,r_2,s)}
\frac{\sin(\frac{\theta}{\rho})}{\cosh(\frac s\rho)-\cos\varphi}a_\lambda\left(D(r_1,r_2,s)\right) \,ds
\]
maps $L^2 \to L^6$ with a gain of $O(\lambda^{-N})$ for any $N>0$ using the decay in $a_\lambda (D)$ in this region.

It thus suffices to consider the contribution of
\[
\int_{c_0(r_1r_2)^{-\frac 12}}^\infty e^{i\lambda D(r_1,r_2,s)}
\frac{\sin(\frac{\theta}{\rho})}{\cosh(\frac s\rho)-\cos\varphi}a_\lambda\left(D(r_1,r_2,s)\right) \,ds.
\]
Since
\[
\prtl_s D(r_1,r_2,s) =  \frac{r_1r_2\sinh s}{D(r_1,r_2,s)}
\]
the phase function has no critical points in $[c_0(r_1r_2)^{-\frac 12},\infty)$.  Using that $D \times a_\lambda(D)$ is bounded, the integral is $O(\lambda^{-1}(r_1r_2)^{-\frac 12})$, and by the argument in \eqref{spcontribution}, this yields a kernel which maps $L^2 \to L^6$ with a gain of $\lambda^{-\frac 56} = \lambda^{-\frac13 - \frac12}$.

\section*{References}
\label{sec:references}

\end{document}